\documentclass[12pt]{amsart}
\usepackage{amsfonts,newlfont,latexsym}
\usepackage{euscript}
\usepackage{amssymb,amscd}
\usepackage{color}

     \theoremstyle{plain}
     \newtheorem{theorem}{Theorem}[section]
     \newtheorem{lemma}[theorem]{Lemma}
     \newtheorem{Lemma}[theorem]{Lemma}
     \newtheorem{corollary}[theorem]{Corollary}
     \newtheorem{proposition}[theorem]{Proposition}

\newtheorem{Theorem}{Theorem}[section]
\newtheorem{Proposition}[Theorem]{Proposition}
\theoremstyle{definition}
\newtheorem{definition}[theorem]{Definition}
\newtheorem{example}[theorem]{Example}

\newcommand{\R}{\mathbb R}
\newcommand{\Rk}{\mathbb R^k}
\newcommand{\C}{\mathbb C}
\newcommand{\Z}{\mathbb Z}
\newcommand{\Zk}{\mathbb Z^k}
\newcommand{\Q}{\mathbb Q}
\newcommand{\T}{\mathbb T}
\newcommand{\Tk}{\mathbb T^k}

\newcommand{\Rm}{\mathbb R^m}
\newcommand{\Rmx}{\mathbb R^m_x}

\newcommand{\Ci}{C^{\infty}}
\newcommand{\CI}{C^{\infty}}

\def \ta{{\tilde \alpha}}
\def \tr{{\tilde \rho}}
\def \tp{{\tilde \phi}}
\def \tm{{\tilde \mu}}
\def \tl{{\tilde \lambda}}

\def \a{\alpha}

\def \G{\Gamma}

\def \e{\varepsilon}

\def \A{(\EuScript{A})}
\def \M{{M}}
\def \S{\cal{S}}
\def \r{\cal{R}}

\def \dim{\mbox{dim}\,}
\def \d{\mbox{dist}\,}

\def \ker{\mbox{ker}}

\def \d{\mbox{dist}}

\def\Proof{{\em Proof}\,: }

\def\QED{~\hfill~ $\diamond$ \vspace{7mm}}

\def \A{\cal A}
\def \Rk {{\mathbb R}^k}
\def \rk {{\mathbb R}^k}
\def \Rm {{\mathbb R} ^m}

\def \T {{\mathbb T}}
\def \C {{\mathbb C }}
\def \c{{\cal C}}

\def \o{{\cal O}}
\def \R{{\mathbb R}}
\def \Z{{\mathbb Z}}
\def \Zk{{\mathbb Z} ^k}
\def \Q{{\mathbb Q}}
\def \la{{\lambda}}

\def \be{{\bar E}}

\def \w{{\cal W}}
\def \ws{{\cal W} ^s}

\def \ci{C^{\infty}}

\definecolor{rjs}{rgb}{.0,.5,.9}
\definecolor{orange}{rgb}{1,0.5,0}

\definecolor{my}{rgb}{0.8,0.4,0.6}

\begin{document}
\author[David Fisher, Boris Kalinin,  Ralf Spatzier]
{David Fisher, Boris Kalinin,  Ralf Spatzier}

\title[Totally non-symplectic Anosov actions]
{Totally non-symplectic Anosov actions on tori and nilmanifolds}

\thanks{First author supported in part by NSF grant DMS-0643546. Second author supported in part by NSF grant DMS-0701292.  Third author supported in part by NSF grants DMS-0906085  and DMS-0604857.}

\address{Department of Mathematics,
Indiana University, Bloomington, IN 47405}

\email{fisherdm@indiana.edu}

 \address{Department of Mathematics \& Statistics,
 University of South Alabama, Mobile, AL 36688}

 \email{kalinin@jaguar1.usouthal.edu}

\address{Department of Mathematics, University of Michigan, Ann Arbor,
MI 48109.}

\email{spatzier@umich.edu}

\date{\today}

\maketitle

\begin{abstract}

We show that sufficiently irreducible totally non-symplectic Anosov actions of higher rank abelian groups on tori and nilmanifolds are  $\ci$-conjugate to
actions by affine automorphisms.

\end{abstract}

\section{Introduction}

Hyperbolic actions of abelian groups of rank at least 2 exhibit many surprising rigidity properties.  Case in point is the local smooth rigidity of actions by automorphisms of tori and nilmanifolds and other algebraically defined actions.     This  means that perturbations of an action that are $C^1$-close for a finite set of generators are $\ci$-conjugate to the original action.   It was
established for algebraic actions with semisimple linear part  by Katok and Spatzier in
\cite{KS94}  and for some non-semisimple actions on tori by Einsiedler and T. Fisher \cite{EiFi}.  The higher rank situation  is entirely different from the case of single
Anosov diffeomorphisms and flows for which it is always easy to construct $C^1$-small perturbations which are not even  $C^1$-conjugate.

Local smooth rigidity of algebraic actions gives strong support  to the following conjecture by Katok and Spatzier.
\vspace{1em}

{\sl Classification Conjecture: All ``irreducible''  Anosov $\Z ^k$ and $\R ^k$-actions for $k \geq 2$ on any compact manifold
are $\ci$-conjugate to algebraic actions.}
\vspace{1em}

Kalinin and Spatzier proved this conjecture for the special class of Cartan actions of abelian groups of rank at least 3 under some other more technical hypotheses \cite{KaSp}.   Here we call an action {\em Cartan} if maximal intersections of stable manifolds of various elements, called {\em coarse Lyapunov foliations}, are one-dimensional and, together with the orbit, span the space. Kalinin and Sadovskaya have results for  more general Anosov actions  of rank at least 2 where the condition on dimension 1 is replaced by  either uniform quasi-conformality or a pinching condition \cite{KaSa3,KaSa4}. The basic  idea of the proofs in all of these results is to build smooth structures  on various foliations and then combine them.  Unfortunately, this only works under strong assumptions on the action.

The general case of the conjecture remains out of reach.  Thus it is natural to restrict attention to actions on tori and nilmanifolds where one usually refers to  the conjecture  as {\em global rigidity}.   For these spaces, the classical results of Franks and Manning \cite{Fr,M} offer a different  approach.   Their work  implies that any action $\alpha$  of an abelian group with at least one Anosov element  on a torus or a nilmanifold is always $C^0$-conjugate to an action by affine Anosov automorphisms by some H\"{o}lder conjugacy $\phi$.   Now to prove global rigidity it suffices to show smoothness of the conjugacy $\phi$.
We call the latter action the {\em linearization} of $\alpha$ and refer to Section 2 for a precise definition.  On the torus the linearization is essentially given by the induced action on homology. Note that in the nilmanifold case, the term ``linearization'' is a bit of a misnomer as the action by automorphisms is not really linear.

The idea that a $C^0$ conjugacy can be used to get $\ci$-rigidity  appears already in Hurder's work  on deformation rigidity of lattice actions on tori \cite{H1} and later in Katok-Lewis \cite{KaLe} for both their local and global rigidity theorems for Cartan actions on tori. It also formed the basis of the argument for local rigidity in Katok-Spatzier \cite{KS94}. In the different context of local rigidity of algebraic actions of lattices in higher rank groups,  work of Katok and Spatzier and later Fisher, Margulis and Qian \cite{FiMa, KS94, MaQi} also involves finding a $C^0$ conjugacy that is improved to $\ci$ using the presence of higher rank abelian subgroups
in the acting group.  Rodriguez Hertz established global rigidity for  $\Zk$ actions on tori with at least one Anosov element whose linearization has
coarse Lyapunov foliations of dimensions one or two and either has maximal rank or
 satisfies additional bunching assumptions \cite{RH}.   To date however,  all results require that the derivatives of either the action or its linearization along the coarse Lyapunov foliations satisfy a pinching assumption.  This means  that the ratio of maximal over minimal contraction is controlled, e.g. less than 2.
In this paper, we overcome this problem for the first time by a combination of the use of  non-stationary normal forms and  holonomy arguments.
 Beyond achieving a superior result, the use of the two tools is also completely novel.
We use limits of holonomy maps to define homogeneous structures on certain foliations. This has never been done before.  Furthermore, we make use of measurable normal forms for the first time in the context of
global rigidity of actions. Previously measurable normal forms have only been used to study invariant measures.

Continuous normal forms were already introduced for the proof of local rigidity in \cite{KS94}.   In essence they give coordinate charts in which  the derivatives of the map along contracting foliations take values in a finite dimensional Lie group.  Moreover, the dependence of the coordinates on the base point is continuous in the $\ci$-topology.  Existence of continuous normal forms  is guaranteed if the derivatives of the maps under consideration satisfy a spectral gap condition along the given contracting foliation.  While such spectral gaps are automatic for $C^1$-perturbations of algebraic  systems and also for one dimensional foliations, they fail to hold in general.  In particular we cannot assume such spectral gaps for the proof of global rigidity. Instead, we use a measurable   version of the  non-stationary normal forms theory where the ``measurable'' spectral gap condition is always satisfied by Oseledec'  Multiplicative Ergodic Theorem.

Let us next summarize some elements from the structure theory of higher rank abelian actions, see Section \ref{preliminaries} for more details.   They preserve a probability measure of full support.  One can find a common Lyapunov splitting of the tangent bundle $TM = \oplus _i  E_i$  which refines the Lyapunov splittings of each individual element.  Moreover,  if $v \in E_i$, the Lyapunov exponent of $v$ defines a linear functional, the {\em Lyapunov functional}, on the acting $\Z^k$ which we think of as a linear functional on the ambient $\R^k$.
For  actions by affine automorphisms
the Lyapunov exponents are nothing but the logarithms of the absolute values of the eigenvalues of the automorphisms.  A {\em Weyl chamber} is a connected component of $\R^k$ minus all the hyperplane kernels of the Lyapunov functionals.   We will need to make the assumption that every Weyl chamber defined by the linearization contains an Anosov element in the non-linear action. As we will later see that the Weyl
chambers on the two sides agree, we abbreviate this by saying   that every Weyl chamber contains an Anosov element. This allows us to define the {\em coarse Lyapunov foliations} as the maximal intersections of stable foliations of Anosov elements.  Hence these foliations are H\"{o}lder with smooth leaves.

Recall that a matrix is semisimple if it is diagonalizable  over $\C$.
We call an action by affine automorphisms of a nilmanifold {\em semisimple}
if the linear part of every element acts by a semisimple matrix.

Finally, we call a $\Z ^k$-action {\em TNS} or {\em totally non-symplectic} if any two $v \in E_i$ and $w \in E_j$ belong to the stable distribution of some element $a \in \Z ^k$.  This excludes the possibility of a bilinear form invariant under the action, hence the name.

The main result of this paper  proves global rigidity for  totally non-symplectic actions.


\begin{theorem}\label{nilmanifold}  Suppose $\alpha $ is a $\ci$-action of $\Z ^k$, $k \geq 2$ on a nilmanifold $N/\Gamma$.  Assume the linearization  $\rho$ of $\alpha$   is  semisimple and TNS and there is an Anosov element in each Weyl chamber of $\alpha$.  Then $\alpha$ is $\ci$-conjugate to $\rho$.
\end{theorem}

As discussed above, this theorem is the first that does not require pinching  conditions.  Moreover,  it also yields  the first global rigidity result for Anosov actions on nilmanifolds  which are not tori.  Indeed, in all earlier results the pinching condition, together with various
additional assumptions such as integrability or absence of certain resonances, forced the nilmanifold to be a torus.

Call a linear $\mathbb Z^k$ action on a torus {\em totally reducible} if
every rational invariant torus has a rational invariant complement.   There is a similar though more complicated notion for nilmanifolds which we describe below in section \ref{section:generalresults}.
 We will show that total reducibility is equivalent to semisimplicity,
and thus we immediately get the next result:

\begin{corollary}
 Suppose $\alpha $ is a $\ci$-action of $\Z ^k$, $k \geq 2$ on a nilmanifold $N/\Gamma$.  Assume the linearization  $\rho$ of $\alpha$   is totally reducible  and TNS and there is an Anosov element in each Weyl chamber of $\alpha$.
 Then $\alpha$ is $\ci$-conjugate to  $\rho$.
 \end{corollary}

To prove the corollary from the theorem, we prove that any totally
reducible action is semisimple.

Our results have some applications to global rigidity for actions of
higher rank lattices.  Margulis and Qian prove that any
Anosov action of a higher rank lattice $\G$ on a nilmanifold with a
common fixed point for the entire group action is continuously
conjugate to an action by affine automorphisms\cite{MaQi}. It is well known that
 such $\G$ contains many abelian subgroups isomorphic
 to $\Z^k$, where $k$ is the real rank of
$\G$, and that the Anosov $\G$ action restricts to an Anosov $\Z^k$
action. If  some $\Z^k$ subgroup satisfies the
conditions of Theorem \ref{nilmanifold}, it then follows from our
results that the conjugacy is smooth, and therefore that the full
$\G$ action is smoothly conjugate to an action by affine automorphisms.

Let us briefly indicate the main elements in the proof of Theorem \ref{nilmanifold}.  As discussed above we show that the topological conjugacy $\phi$ is smooth.  For this,
we first suspend the $\Z^k$-action to an $\R^k$-action.  Then we fix a coarse Lyapunov
foliation and for almost every leaf we construct a transitive group of smooth transformations
which is intertwined by $\phi$ with the group of translations of the corresponding leaf
for the linearization.  As in other proofs of rigidity theorems e.g. in \cite{KS94}, we use limits of return maps.  Unlike earlier proofs however, we do not directly use the acting group but rather holonomies along transversal coarse Lyapunov foliations.  First  we show that  these holonomies are smooth. For this we establish existence of elements which contract the fixed coarse Lyapunov foliation slower than a
transversal one. Then we show that the holonomies centralize suitable elements of $\R^k$ and hence preserve measurable non-stationary normal forms. It follows that limits of
such holonomies are still smooth and define the desired transitive group actions.
Once the smoothness of $\phi$ is established for a.e. leaf of each coarse Lyapunov
foliation, the smoothness of holonomies gives the global smoothness of $\phi$.
A more detailed outline of the proof is given in Section \ref{outline}, after all
relevant notions have been defined.

We would like to thank K. Burns, D.Dolgopyat, F.Ledrappier Y. Pesin, J. Rauch and A. Wilkinson for a number of  discussions on subjects related to this paper.












\section{Preliminaries}   \label{preliminaries}

Throughout the paper, the smoothness of diffeomorphisms, actions,
and manifolds is assumed to be $\Ci$, even though all definitions and
some of the results can be formulated  in lower regularity.

\subsection{Anosov actions of $\,\Zk$ and $\Rk$}

\label{Anosov actions} $\;$
\vskip.1cm

Let $a$ be a diffeomorphism of a compact manifold $\M$.
We recall that  $a$ is  {\em Anosov} if there exist a continuous $a$-invariant
decomposition  of the tangent bundle $T\M=E^s_a \oplus E^u_a$ and constants
$K>0$, $\lambda>0$ such that for all $n\in \mathbb N$
\begin{equation} \label{anosov}
  \begin{aligned}
 \| Da^n(v) \| \,\leq\, K e^{- \lambda n} \| v \|&
     \;\text{ for all }\,v \in E^s_a, \\
 \| Da^{-n}(v) \| \,\leq\, K e^{- \lambda n}\| v \|&
     \;\text{ for all }\,v \in E^u_a.
 \end{aligned}
 \end{equation}
The distributions $E_a^s$ and $E_a^u$ are called the {\em stable} and {\em unstable}
distributions of  $a$.
\vskip.2cm

Now we consider a $\Zk$ action $\alpha$ on a compact manifold $\M$
via diffeomorphisms.
The action is called {\em Anosov}$\,$ if there is an element which acts as an
Anosov diffeomorphism. For an element $a$ of the acting group we denote
the corresponding diffeomorphisms by $\a (a)$ or
simply by $a$ if the action is fixed.

For a $\Zk$ action $\alpha$ on a  manifold $M$, there is an associated
$\Rk$ action $\ta$ on a manifold $\S$ given by the standard suspension
construction \cite{KaK}.   Briefly, this is the action of $\R ^k$ by left translations on
$(\R^k \times M)/{\Z ^k}$.  Here $(\R^k \times M)/{\Z ^k}$ is the quotient of $\R^k \times M$ by the $\Z^k$-action on $\R^k \times M$ given by $z (r,p) = (r -z, z (p))$. We will refer to $\ta$ as the {\em suspension} of $\a$.
It generalizes the suspension flow of a diffeomorphism. Similarly,
the manifold $\S$ is a fibration over the ``time" torus $\Tk$ with fiber $\M$.

\begin{definition} Let $\alpha$ be a smooth  action of $\,\Rk$
on a compact manifold $\M$. An element $a \in \Rk$ is called {\em Anosov}
or {\em normally hyperbolic} for $\alpha$ if there exist positive constants $\lambda$,
$K$ and a continuous $\alpha$-invariant splitting of the tangent bundle
$$
T\M = E^s _a \oplus E^u _a \oplus T\o
$$
where $T\o$ is the tangent distribution of the $\Rk$-orbits, and \eqref{anosov} holds
for all $n\in \mathbb N$.
\end{definition}

An $\Rk$ action is called {\em Anosov} if some element $a \in \Rk$ is Anosov.
Note that  $a \in \Zk$ is Anosov for $\a$ if and only if it is Anosov for $\ta$. Thus if
$\a$ is an Anosov $\Zk$ action then $\ta$ is an Anosov $\Rk$ action.

Both in the discrete and the continuous case it is well-known that the distributions
$E_a^s$ and $E_a^u$  are H\"{o}lder continuous and tangent
to the stable and unstable foliations $\w_a^s$ and $\w_a^u$ respectively \cite{HPS}.
The leaves of these foliations are $C^\infty$ injectively immersed Euclidean
spaces. Locally, the immersions vary continuously in the $\Ci$ topology.  In general, the
distributions $E^s$ and $E^u$ are only H\"older continuous transversally
to the corresponding foliations.

\subsection{Lyapunov exponents and coarse Lyapunov distributions}
\label{Lyapunov} $\;$
\vskip.1cm

First we recall some basic facts from the theory of non-uniform
hyperbolicity for a single diffeomorphism, see for example \cite{BP}. Then
we consider $\Zk$ and $\Rk$ actions concentrating on the
continuous time case on the case, we refer to \cite{KaSp} and
\cite{KaSa3} for more details.

Let  $a$ be a diffeomorphism of a compact manifold $\M$ preserving an
ergodic  probability measure $\mu$. By Oseledec' Multiplicative Ergodic
Theorem, there exist finitely many numbers $\chi _i $ and an invariant  measurable
splitting of the tangent bundle $T\M = \bigoplus E_i$ on a set of full measure
such that the forward and backward Lyapunov exponents of $v \in  E_i $
are $\chi _i $.  This splitting is called {\em Lyapunov decomposition}.
 We define the stable distribution of $a$ with respect to $\mu$
as $E^-_a= \bigoplus _{\chi _i <0}  E_i$. The subspace $E^-_a(x)$
is tangent $\mu$-a.e. to the stable manifold $W^-_a(x)$.
More generally, given any $\theta <0$ we can define the strong
stable distribution by $E^{\theta}_a= \bigoplus _{\chi _i \le \theta}  E_i$
which is tangent $\mu$-a.e. to the strong stable manifold $W^{\theta}_a(x)$.
$W^{\theta}_a(x)$ is a smoothly immersed Euclidean space. For a
sufficiently small ball $B(x)$, the connected component of
$W^{\theta}_a(x) \cap B(x)$, called {\em local} manifold, can
be characterized by the exponential contraction property
\begin{equation} \label{loc}
W^{\theta,loc}_a(x) = \{y \in B(x) \, | \; \d(a^n x, a^n y) \le C e^{(\theta +\e)n}
\quad \forall n\in \mathbb N \}.
\end{equation}
The unstable distributions and manifolds are defined similarly.
In general, $E^-_a$ is only measurable and depends on the measure $\mu$.
However, if $a$ is an  Anosov diffeomorphism, or an Anosov element of
an $\Rk$ action, then $E^-_a$ for any measure always agrees with the
continuous stable distribution $E^s_a$. Indeed, $E^s_a$ cannot contain
a vector with a nontrivial component in some $E_j$ with $\chi _j \geq 0$
since such a vector does not satisfy \eqref{anosov}.  Hence $E^s_a \subset \bigoplus _{\chi _i <0}  E_i$.  Similarly, the unstable distribution $E^u_a \subset \bigoplus _{\chi _i  > 0}  E_i$.  Since $E^s_a \oplus E^u_a $ is transverse
to the orbit, both inclusions have to be equalities.

Let $\mu$ be an ergodic probability measure for an $\Rk$ action $\alpha$
on a compact manifold $\M$. By commutativity, the  Lyapunov decompositions
for  individual elements of $\Rk$ can be refined to a joint invariant splitting for
the action. The following proposition from \cite{KaSp} describes the
Multiplicative Ergodic Theorem for this case. See  \cite{KaSa3} for the discrete
time version and \cite{KaK} for more details on the Multiplicative Ergodic
Theorem and related notions for higher rank abelian actions.

\begin{proposition} Let $\a$ be a smooth action of  $\Rk$ and let $\mu$ be an
ergodic invariant  measure.
There are finitely many linear functionals $\chi $ on $\Rk$, a set of
full measure ${\cal P}$, and an $\alpha$-invariant  measurable splitting
of the tangent bundle  $T\M =  T\o \oplus \bigoplus E_{\chi}$ over ${\cal P}$, where $\o$
is the orbit foliation, such that
 for all $a \in \Rk$ and $ v \in E_{\chi}$, the Lyapunov exponent of $v$ is $\chi
 (a)$, i.e.
$$
  \lim _{t \rightarrow \pm \infty }
  t^{-1} \log \|\big( D \a (ta) \big) (v) \| = \chi (a),
$$
 where $\| .. \|$ is a continuous norm on $T\M$.
\end{proposition}

The splitting  $ \bigoplus E_{\chi} $ is called the {\em Lyapunov decomposition},
and the linear functionals $\chi$ are called the {\em Lyapunov exponents} of
$\alpha$. The hyperplanes $\,\ker \,\chi \subset \Rk$
are called the {\em Lyapunov hyperplanes} or {\em Weyl chamber walls}, and the
connected components of
$\,\Rk - \cup _{\chi} \ker \chi $ are called the {\em Weyl chambers} of $\alpha$.
The elements in the union of the Lyapunov hyperplanes are called
{\em singular}, and the elements in the union of the Weyl chambers
are called {\em regular}.
We note that the corresponding notions for a $\Zk$ action and for its suspension are
directly related. In particular, the nontrivial Lyapunov exponents are the same.
In addition, for the suspension there is one identically zero Lyapunov exponent
corresponding to the orbit distribution. From now on, the term Lyapunov exponent
will always refer to the nonzero functionals.

Consider  a $\Zk$ action by {\em automorphisms} of a torus or a nilmanifold
$M=N/\G$.
In this case, the Lyapunov decomposition is determined by the eigenspaces of the
automorphisms, and the Lyapunov exponents are the logarithms of the moduli
of the eigenvalues. Hence they are independent of the invariant measure, and
they give uniform estimates of expansion and contraction rates. Also, every Lyapunov
distribution is smooth and integrable.

In the non-algebraic case, the individual Lyapunov distributions are in general
only measurable and depend on the given measure. This can be already
seen for a single diffeomorphism, even if Anosov. However, as we observed above,
the full stable distribution $E^s_a$ of an Anosov element $a$  always
agrees with $\bigoplus _{\chi  (a) <0}  E_\chi$ on a set of full measure for any  measure.

For higher rank actions, {\em coarse Lyapunov distributions} play a similar role to the stable and unstable distributions for an Anosov diffeomorphism.
For any Lyapunov functional $\chi$ the coarse Lyapunov distribution is the direct
sum of all Lyapunov spaces with Lyapunov functionals positively proportional to $\chi$:
$$
   E^{\chi} = \oplus E_{\chi '}, \quad \chi ' = c \, \chi \,\text{ with }\, c>0.
$$

For an algebraic action such a distribution is a finest nontrivial
intersection of the stable distributions of certain Anosov elements of the action.
For nonalgebraic actions, however, it is not a priori clear. It was shown in \cite[Proposition 2.4]{KaSp} that, in the presence
of sufficiently many Anosov elements, the coarse Lyapunov distributions are
well-defined, continuous, and tangent to foliations with smooth leaves
(see Proposition 2.2 in \cite{KaSa4} for the discrete time case).
We denote the set of all Anosov elements in $\Zk$ or $\rk$ by $\A$.

\begin{proposition} \label{CoarseLyapunov}
Let $\alpha$ be an Anosov action of $\Zk$ or $\Rk$ and let $\mu$ be an ergodic  probability
measure for $\a$ with full support. Suppose that there exists an Anosov element
in every Weyl chamber defined by $\mu$. Then for each Lyapunov exponent
$\chi $ the coarse Lyapunov distribution can be defined as
$$
  E^{\chi}(p) =  \bigcap _{\{a \in \A\, | \; \chi (a) <0\}} E^s _a(p) \;=
  \bigoplus_ {\{ \chi ' = c \, \chi\,|\;  c>0 \}}  E_{\chi '} (p)
 $$
on the set ${\cal P}$ of full measure where the Lyapunov exponents exist.
Moreover, $E^{\chi}$ is H\"{o}lder continuous, and thus it can
be extended to a H\"{o}lder distribution tangent to the foliation
$\w ^\chi = \bigcap  _{\{a  \in \A \,  \mid \, \chi (a) <0\}} \ws _a$
with uniformly $\Ci$ leaves.
\end{proposition}

Note that ergodic measures with full support always exist if a $\Zk$ action
contains a transitive Anosov element.

A natural example is given by the
measure $\mu$ of maximal entropy for such an element, which is unique
 \cite[Corollary 20.1.4]{HaKa} and hence is invariant under the whole action.

Since a coarse Lyapunov distribution is defined by a collection of
positively proportional Lyapunov exponents it can be uniquely identified
with the corresponding positive (negative) set of these functionals,
called the {\em positive (negative) Lyapunov half-space}, or with the
oriented Lyapunov hyperplane that separates them.

The action is called {\bf totally nonsymplectic}, or {\bf TNS}, if there are
no negatively proportional Lyapunov exponents. Equivalently, any two
different negative Lyapunov half-spaces have nontrivial intersection.
Therefore, any pair of coarse Lyapunov distributions for such an action
is contracted by the elements in this intersection.

\subsection{$\Zk$ and $\Rk$ actions on tori and nilmanifolds.}


\vskip.1cm

Let $f$ be an Anosov diffeomorphism of a torus or, more generally,
a nilmanifold $M=N/\G$. By the results of Franks and Manning in \cite{Fr,M},
$f$ is topologically conjugate to an Anosov automorphism $A : M \to M$,
i.e. there exists a homeomorphism  $\phi : M \to M$ such that $A \circ \phi = \phi \circ f$. The conjugacy $\phi$ is  bi-H\"older, i.e. both $\phi$ and $\phi^{-1}$
are H\"older continuous with some H\"older exponent $\gamma$.

Now we consider an Anosov $\Zk$ action $\a$ on a nilmanifold $M$.
Fix an Anosov element $a$ for $\a$. Then we have $\phi$ which
conjugates $\a(a)$ to an automorphism $A$. By \cite[Corollary 1]{W}
any homeomorphism of $M$ commuting with $A$ is an affine automorphism.
Hence we conclude that $\phi$ conjugates $\a$ to an action $\rho$
by affine automorphisms. We will call $\rho$ an {\em algebraic action}
and refer to it as the {\em linearization} of $\a$.

Now we describe the preferred invariant measure for $\a$
(cf.  \cite[Remark 1]{KaK07}).
We denote by $\lambda$ the normalized Haar measure on the
nilmanifold $M$. Note that $\lambda$ is invariant under any
affine automorphism of $M$ and is the unique measure of
maximal entropy for any affine Anosov automorphism.

\begin{proposition} \label{mu}
The action $\a$ preserves an absolutely continuous measure
$\mu$ with smooth positive density. Moreover,
$\mu = \phi^{-1} _\ast (\lambda)$ and for any Anosov element
$a\in \Zk$, $\mu$ is the unique measure of maximal entropy for $\a(a)$.
\end{proposition}

\begin{proof}
Let $J$ denote  the Jacobian of $\alpha$ with respect to the
Haar measure $\lambda$, more precisely, for any $b \in \Z^k$ the
function $J_b (x)$  denotes the density of the push forward measure
$\a(b)_\ast (\lambda )$ with respect to $\lambda$.
Since the conjugacy $\phi$ is bi-H\"{o}lder, $\log \,(J \circ \phi^{-1})$
is a H\"{o}lder cocycle over the linearization $\rho$.
By rigidity of H\"{o}lder cocycles
for irreducible algebraic $\Z ^k$-actions \cite{KS94},
this cocycle is H\"{o}lder cohomologous to a constant one, i.e.
there exists a linear functional $c :\Z^k \to \R$ and a H\"{o}lder
continuous function
$\Phi$ on $M$ such that for all $b \in \Z^k$ and $x \in M$
$$
\log \,(J_b \circ \phi^{-1}(x))  = c(b) + \Phi (\rho (b)x) - \Phi (x) .
$$
a constant $c$. Hence for the function $\Psi =\Phi \circ \phi$ we have
$$
\log \,(J_b (x) ) = c(b) + \Psi (\a (b)x) - \Psi (x) .
$$
Let $\mu$ be the measure $e^{-\Psi(x)} \lambda$ normalized by
$\mu (M)=1$. It follows that the Jacobian $\tilde J$ of $\alpha$ with
respect to $\mu$ is constant,  $\tilde J_b (x) = e^{c(b)}$. Since
$\a(b)$ is a diffeomorphisms we have $\int_M \tilde J_b (x) d\mu =1$
and hence $e^{c(b)}=1$. Thus $\a$ preserves $\mu$.

Let $a$ be an Anosov element of $\a$. Then the density of its
absolutely continuous measure $\mu$ is $C^1$ \cite[Theorem 19.2.7]{HaKa},
and in fact $C^\infty$  \cite[Corollary 2.1]{LMM}. Also, $\mu$
is the equilibrium state of $\log \,(J^u)$ \cite[Theorem 20.4.1]{HaKa},
where $J^u$ is the Jacobian of $\a(a)$ along its unstable distribution.
As before, we have that $\log \,(J^u)$ is cohomologous to a constant
function. Since the measure of maximal entropy of the transitive Anosov
diffeomorphism $\a(a)$ is the equilibrium state of the constant function
\cite[Theorem 20.1.3]{HaKa}, and since the equilibrium state is unique
\cite[Theorem 20.3.7]{HaKa} we conclude that $\mu$ is the measure
of maximal entropy for $\a(a)$. Since $\lambda$ is the measure
of maximal entropy for $\rho(a)$ then by conjugation so is the measure
$\phi^{-1} _\ast (\lambda)$ for $\a(a)$. Hence by uniqueness we have
$\mu = \phi^{-1} _\ast (\lambda)$.
\QED
\end{proof}

We will show below that the Lyapunov exponents of $(\a, \mu)$ and
$(\rho, \lambda)$ are positively proportional and that the corresponding
coarse Lyapunov foliations are
mapped into each other by the conjugacy $\phi$.

We consider the suspensions $\ta$ and $\tr$ of $\a$ and $\rho$.
These are smooth
$\Rk$ actions on the  suspension manifolds $\S$ and $\r$ of $\a$ and
$\rho$. We denote the lifts to the suspensions of the conjugacy and the invariant
measures by $\tp$, $\tm$, and $\tl$. Note that $\tp$ and $\tp ^{-1}$ are also H\"older continuous with  the same exponent $\gamma >0$ as $\phi$ and $\phi ^{-1}$.

 From now on, instead of indexing a coarse Lyapunov by a representative of the class of positively proportional
Lyapunov functionals, we index them numerically.  I.e. we write $\w^i$ instead of $\w^{\chi}$, implicitly identifying
the finite collection of equivalence classes of Lyapunov exponents with a finite set of integers.
The next proposition summarizes important properties of the suspension actions.
Similar properties hold for the original $\Zk$ actions.


\begin{proposition} \label{properties}
 Assume there is an Anosov element in every  Weyl chamber. Then

\begin{enumerate}

\item The Lyapunov exponents of $(\ta, \tm)$ and $(\tr, \tl)$ are positively
proportional, and thus the Lyapunov hyperplanes and Weyl chambers are
the same.

\item For any coarse Lyapunov foliation $\w^i_\ta$ of $\ta$
$$
\tp (\w^i_\ta ) = \w^i_\tr ,
$$
where $\w^i_\ta$ is the corresponding coarse Lyapunov foliation for $\tr$.

\end{enumerate}

\end{proposition}

{\bf Remark.} We do not claim  at this point that the Lyapunov exponents
of $(\ta, \tm)$ and $(\tr, \tl)$ (or of different invariant measures for $\ta$)
are equal. Of course if $\ta$ is shown to be smoothly conjugate to
$\tr$ then this is true a posteriori.

{\bf Remark.} In fact, one can show that the same holds for Lyapunov exponents and
coarse Lyapunov foliations of $(\a, \nu)$ for any $\a$-invariant measure $\nu$ so,
in particular, the Lyapunov exponents of all $\a$-invariant measures
are positively proportional and the coarse Lyapunov splittings are consistent with
the continuous one defined in Proposition \ref{CoarseLyapunov}.

\begin{proof}
First we observe that the  conjugacy $\tp$ maps the stable
manifolds of $\ta$ to those of $\tr$. More precisely, for any $a \in \Rk$
and any for $\mu$-a.e. $x \in \S$ we have
\begin{equation} \label{phimapsW}
\tp (W^-_{\ta (a)}(x)) = W^-_{\tr (a)}(\tp(x)).
 \end{equation}
Indeed, it suffices to establish this for local manifolds, which are
characterized by the exponential contraction as in \eqref{loc}.
Since $\tp$ is bi-H\"older, it preserves the  property that $\d (x_n, y_n)$
decays exponentially, which implies \eqref{phimapsW}.
In particular, for any Anosov $a \in \Rk$ and any $x \in \S$
we have $\tp (W^s_{\ta (a)}(x)) = W^s_{\tr (a)}(\tp(x))$.
Hence the formula for $\w^i_\ta$ given in Proposition
\ref{CoarseLyapunov} implies (2) once we establish (1).

To establish (1)  it suffices to show that the oriented Lyapunov hyperplanes
of $(\ta, \tm)$ and $(\tr, \tl)$ are the same. Suppose that an oriented Lyapunov
hyperplane $L$ of one action, say $\ta$, is not an oriented Lyapunov hyperplane
of the other action $\tr$.  We take an element $a$ close to $L$ in the corresponding positive Lyapunov half-space $L^+$ and denote the reflection
of $a$ across $L$ by $b$. We can choose them so that $a$ and $b$ are
not separated by any other Lyapunov hyperplane of either action.
Then, $E^-_{\ta (b)} = E^-_{\ta (a)} \oplus E$, where $E$ is the
coarse Lyapunov distribution of $\ta$ corresponding to $L$.
Similarly, since we assumed that $L^+$ is  not a positive Lyapunov
 half-space for $\tr$, we have  $E^-_{\tr (b)} \subseteq E^-_{\tr (a)}$.
 We conclude that

$$
W^-_{\ta (a)}  \subsetneq W^-_{\ta (b)} \quad \text{ but } \quad
W^-_{\tr (a)}  \supseteq W^-_{\tr (b)},
$$
which contradicts \eqref{phimapsW} since $\tp$ is a homeomorphism.
 \QED
\end{proof}


\section{Outline of the proof of Theorem \ref{nilmanifold}} \label{outline}

Proposition \ref{CoarseLyapunov} shows that coarse Lyapunov foliations for
$\a$ and $\ta$ are well-defined continuous foliations with smooth leaves.
By Proposition \ref{properties} they are mapped by the conjugacy
to the corresponding homogeneous foliations for $\rho$ and $\tr$. The main goal
is to study the regularity of the conjugacy $\phi$ along these foliations.

For the most of the proof we consider a coarse Lyapunov foliation $\w$ of
the suspension action $\ta$. The first major step is to establish
smoothness of certain holonomies between leaves of $\w$. The TNS assumption
gives the existence of invariant foliations $\w_1$ and $\w_2$ such that
$T\w_1 \oplus T\w \oplus T\w_2 \oplus T\o= TM$.  Moreover, each $T\w_i \oplus T\w$ is
the stable distribution of some element and, in particular, is integrable.
In Section \ref{Smooth holonomies} we show that the holonomies along
$\w_1$ (and along $\w_2$) between leaves of $\w$ are $\ci$. This follows
from the existence of an element for which $\w_1$ a fast stable foliation inside
$T\w_1 \oplus T\w$. To obtain such an element we establish in
Section \ref{uniform estimates} that the expansion or contraction of $\w$ by an
element in the corresponding Lyapunov hyperplane is uniformly slow.

The second major step is to establish smoothness of the conjugacy $\phi$ along
the leaves of  the coarse Lyapunov foliation $\w$. For this we introduce in
Section \ref{section:normalforms} the measurable normal forms for the action
on $\w$ defined almost everywhere with respect to the measure $\mu = \phi^{-1}_\ast (\la)$.
In Section \ref{commuting holonomies} we show that the smooth holonomies
along $\w_1$ preserve the normal forms on $\w$. For this we use the
semisimplicity assumption to split the homogeneous foliation $\tp (\w_1)$ into
subfoliations corresponding to eigenspaces of $\rho$. Then we see that holonomies
along a particular subfoliation preserve the normal forms since they commute with
an element in $\rk$ which fixes the corresponding eigenspace and contracts $\w$.
Since $\w_1$ is the full stable foliation of some element, it is ergodic with respect
to $\mu$, and hence the holonomies along a typical leaf are sufficiently transitive.
Using this we show in Section \ref{limit argument} that for a typical leaf $W$ of $\w$ and for almost every translation $T$ of the homogeneous leaf $\phi (W)$,
the conjugate map $\phi^{-1} \circ T \circ \phi : W \to W$ can be obtained
as a certain limit of such holonomies.  Then this map also preserves the
normal forms and therefore is smooth. This yields that $\phi$ is $\ci$ along $W$.

Since the holonomies between different leaves of $\w$ along $\w_1$ and $\w_2$
 are smooth and intertwine the restriction of $\phi$ to these leaves
we obtain that $\phi$ is $\ci$ along all leaves of $\w$ and that the derivatives are
continuous transversally. Then standard elliptic theory implies that $\phi$ is
$\ci$ on $M$.

\section{Uniform estimates for elements near Lyapunov hyperplane}
\label{uniform estimates}

We consider the suspension actions $\ta$ and $\tr$ of $\rk$ on $\S$  and $\r$.
We fix a Lyapunov hyperplane $L \subset \rk$ and the corresponding
positive Lyapunov half-space $L^+$. We denote the corresponding
coarse Lyapunov distributions for $\ta$ and $\tr$ by $E$ and $\be$
respectively. Recall that $\gamma >0$ denotes a H\"older exponent
of $\tp$ and $\tp ^{-1}$.

\begin{lemma} \label{upper est}
Consider an element $b \in \rk$. Let $\bar \chi (b)$ be the largest Lyapunov exponent
of $\tr (b)$ corresponding to $\be$ and denote $\chi_M =\max \{ 0, \bar \chi (b)/ \gamma \}$.
Let $\nu$ be {any} ergodic invariant measure for $\ta (b)$ and let $\chi_\nu (b)$ be the
largest Lyapunov exponent of $(\ta (b), \nu)$ corresponding to the distribution $E$. Then
$\chi_\nu (b)\le \chi_M$
\end{lemma}

\begin{proof} Suppose that $\chi_\nu (b) > \chi_M$.
Let $E^{uu}$ be the distribution spanned by the Lyapunov subspaces of $(\ta (b), \nu)$ corresponding to Lyapunov exponents greater than $\chi_M+\e$. Then, for some $\e >0$,
$E^{uu}$ has nonzero intersection with the distribution $E$.
The strong unstable distribution $E^{uu}(x)$
is tangent for $\nu$-a.e. $x$ to the corresponding strong unstable manifold $W^{uu}(x)$.
Hence the intersection $F(x)$ of $W^{uu}(x)$ with the leaf $W(x)$ of the coarse Lyapunov
foliation corresponding to $E$ is a submanifold of positive dimension.
Take $y \in F(x)$ and denote $y_n=\ta(-nb)(y)$ and $x_n=\ta(-nb)(x)$.
Then $x_n$ and $y_n$ converge exponentially with the rate at least $\chi_M+\e$.
Since the conjugacy $\tp$ is $\gamma$ bi-H\"older it is easy to see that
$$\d (\tp (x_n),\tp (y_n)) = \d (\tr (-nb)(x),\tr (-nb)(y))$$
decreases at a rate faster than $\gamma \, \chi_M$. But this is impossible since $\tp$
maps $W(x)$ to the corresponding foliation of the linearization which
is contracted by $\tr (-b)$ at a rate at most $\gamma \, \chi_M$. \QED
\end{proof}

\begin{proposition} \label{e-slow}
Let $L \subset \rk$ be a Lyapunov hyperplane and $E$ be the corresponding
coarse Lyapunov distribution for $\ta$. For any $\e>0$ there exist $\eta>0$
so that for any element $b \in \rk$  with $\d \, (b, L) \le \eta \, \e \,$ there
exists $C=C(b,\e)$ such that
\begin{equation}  \label {e-est}
(C e^{\e  n})^{-1} \| v \| \le  \| D (\ta (nb)) v \| \le  C e^{\e n} \| v \|
\quad \text{for all } v \in E, n\in \mathbb N.
\end{equation}

\end{proposition}

\begin{proof}
In the proof we will abbreviate $\ta (b)$ to $b$.
Consider functions $a_n (x) = \log \| D b^n |_E(x) \|$, $n \in \mathbb N$.
Since the distribution $E$ is continuous, so are the functions $a_n$.
The sequence $a_n$ is subadditive, i.e. $a_{n+k} (x) \le a_n (b^k(x)) + a_k (x)$.
The Subadditive and Multiplicative Ergodic Theorems imply that for every
$b$-invariant ergodic measure $\nu$ the limit $\lim _{n\to \infty} \, {a_n(x)/n} $
exists for $\nu$-a.e.$\,x$ and equals the largest Lyapunov exponent of $(b, \nu)$
on the distribution $E$.

The largest exponent $\bar \chi (b)$ of $\tr (b)$ from Lemma \ref{upper est}
can be estimated from above by $c \cdot dist(b,L)$ for some $c>0$. Hence
we can find $\eta>0$ so that the number $\chi_M$ from Lemma \ref{upper est}
is less than $\e/2$ for all $b \in \rk$  with $\d \, (b, L) \le \eta \, \e$. Then Lemma
\ref{upper est} implies that $\lim _{n\to \infty} \, {a_n(x)/n} \le \e/2$ for almost every
$x$ with respect to any $b$-invariant ergodic measure $\nu$.
Thus the exponential growth rate of $\| D b^n |_E(x) \|$ is less than $\e/2$ for
all $b$-invariant ergodic measures. Since $\| D b^n |_E(x) \|$ is continuous,
this implies the uniform exponential growth estimate, as in the second
inequality in \eqref{e-est} (see \cite[Theorem 1]{Sch} or \cite[Proposition 3.4]{RH}).
The first inequality  in \eqref{e-est} can be obtained from the second one for $-b$.
 \QED
\end{proof}

\section{Smooth holonomies.} \label{Smooth holonomies}

We consider the suspension actions $\ta$ and $\tr$ of $\rk$ on $\S$.
We fix a Lyapunov hyperplane $L \subset \rk$ and denote by $E$ and $\w$
the corresponding coarse Lyapunov distribution and foliation for $\ta$ on $\S$.
These are unique by the TNS hypothesis as there are no negatively proportional Lyapunov exponents.
In this section we establish smoothness of certain holonomies between leaves of $\w$.

We first need a technical result on existence of suitable complementary foliations.
\begin{Lemma} \label{splittings}
For a TNS action $\alpha$, suppose that every Weyl chamber contains an Anosov element.  Then there are  $\ta$-invariant distributions $E_1$
and $E_2$ such that $E_1 \oplus E \oplus E_2 \oplus T\o = T\S$.
Moreover, both $E_i $ and $E_i \oplus E$, $i=1,2$ are the stable distribution
of some Anosov elements, and hence  are tangent to invariant foliations
which we denote respectively by $\w_i $ and $\w_i \oplus \w$, $i=1,2$.
\end{Lemma}

\Proof
Consider a generic plane $P$ in $\rk$ which intersects
different Lyapunov hyperplanes in distinct lines $L_i$.  Recall that the TNS
assumption implies that each Lyapunov hyperplane bounds a unique
negative Lyapunov half-space.  Thus the  $L_i$ are naturally oriented, and we can order them   cyclically $L=L_1, L_2, ..., L_n$.  Let $m$ be the index
such that $-L_1$ is between $L_m$ and
$L_{m+1}$. There are two Weyl chambers in the negative Lyapunov half-space
$L_1^-$ whose intersections with the plane $P$ border $L_1$. By assumption,
there exist Anosov elements in these Weyl chambers, which we denote $a_1$
and $a_2$. Similarly, there are two Weyl chambers across $L_1$ in the positive
Lyapunov half-space $L_1^+$. We denote Anosov elements in these Weyl chambers by $c_1$ and $c_2$. More precisely, if we order the Weyl
chambers intersecting $P$ cyclically from $L_1$: $\c _i$, $i=1, ..., n$
then we can take $a_1 \in \c _1$, $a_2 \in \c _m$, $c_2 \in \c _{m+1}$,
$c_1 \in \c _n$.

 Denote the coarse Lyapunov distribution corresponding to
$L_i$ by $E^i$. Note that $E=E^1$.  Then one can see that $T\S = E_1 \oplus E \oplus E_2 \oplus T\o $,
where
$$ E_2 := E^s_{c_2} = E^2 \oplus .... \oplus E^m \qquad
E^s_{a_2} = E^1 \oplus .... \oplus E^m = E \oplus E_2 \qquad \text{and}$$
$$ E_1 := E^s_{c_1} = E^{m+1} \oplus .... \oplus E^n \qquad
E^s_{a_1} = E^{m+1} \oplus .... \oplus E^n \oplus E^1 = E_1 \oplus E
.$$
Since stable distributions of Anosov elements integrate to invariant foliations, the claim follows. \QED

We will show that the holonomies along $\w_i$, $i=1,2$ between leaves
of $\w$ are $\ci$. This follows from the existence of an element which contracts
$\w_1$ (resp. $\w_2$) faster than it does $\w$.

\begin{proposition} \label{fast stable}
In the above notations, for $i=1,2$, there exist elements $b_i \in \rk$ such that
$b_i$ contracts $\w_i$ faster than it does $\w$, i.e.
\begin{equation} \label{domin}
\| D (\ta (b_i)) |_{E_i} \|  <  \| D (\ta (-b_i)) |_E \|^{-1} \le  \| D (\ta (b_i)) |_E \| <1 .
\end{equation}

\end{proposition}
Since the faster part of an (un)stable foliation is $\ci$
inside of an (un)stable leaf, see for example \cite[Proposition 5.1]{KaSa3} or \cite[Proposition 3.9]{KaSa4}, we obtain the following corollary:

\begin{corollary} \label{smooth}
In the above notations, for $i=1,2$, the leaves of $\w_i$ vary smoothly
along the leaves of $\w$, and the holonomies along $\w_i$
between leaves of $\w$ are $\ci$.
\end{corollary}

\begin{proof}  ({\em of  Proposition \ref{fast stable}.)}
We use the notations from the proof of Lemma \ref{splittings}.    We
will first find an element $b'$ close to $L$ which does not expand
or contract $E$ much, with uniform control.  Then a suitable
combination of $a_i$ with $b'$ will suffice.

We consider the case $i=1$ and denote $a=a_1$, $c=c_1$, and $F=E \oplus E_1$. The other cases are similar, and will not be discussed.
We have that $a$ uniformly contracts $F$ and
$c$ uniformly contracts $E_1$, i.e. there exist $C_1,\chi >0$ such that for all $t>0$
\begin{equation} \label{fs1}
\| D (\ta (ta)) v \| \le  C_1 e^{- \chi t} \| v \| \quad  \forall v \in F, \qquad
\| D (\ta (tc)) v \| \le  C_1 e^{- \chi t} \| v \|  \quad \forall v \in E_1
\end{equation}

 Since $E$ is a continuous distribution on $\S$ and $\S$ is compact, $a$ contracts $E$ by at most
 $\sup _{x \in \S}  \| da ^{-1}\mid  _E  (x)\|. $  Hence  there is a fastest contraction rate $\chi'$ for $a$ on
$E$ such that for some $c_2>0$ and all $t>0$
\begin{equation} \label{fs2}
\| D (\ta (ta)) v \| \ge  c_2 e^{- \chi' t} \| v \| \quad  \forall v \in E
\end{equation}

 Since $F \supset E$, equations (\ref{fs1}) and
(\ref{fs2}) imply that $\chi'>\chi$.

Let $b'=ra+(1-r)c$, $0<r<1$, be a convex combination of $a$ and $c$.
Note that by \eqref{fs1} any such $b'$ uniformly contracts $E_1$:
\begin{equation} \label{fs3}
\| D (\ta (tb')) v \| \le  C_1^2 e^{- \chi t} \| v \|  \quad \forall \, v \in E_1, \; \forall \, t>0.
\end{equation}

We will find an element satisfying \eqref{domin}  in the form $b=t(b' + s a)$,
where $t>0$ is large and $s>0$ is small. For any $\e >0$ we can choose $b'$
so that it is in $L^-$ and sufficiently close to $L$ so that Proposition \ref{e-slow}
applies and so $b'$ contracts $E$ very slowly. Then equations \eqref{fs1}, \eqref{fs2}, \eqref{fs3} yield that there exists
$K>0$ such that for all $t>0$
$$
\| D (\ta (b)) v \| \le K e^{-  (\chi+s\chi) t} \| v \| \qquad  \forall v \in E_1,
\qquad \qquad \text{and}
$$
$$
K^{-1} e^{- (s \chi' + \e) t} \| v \| \le \| D (\ta (b)) v \|
\le K e^{- (s \chi - \e) t} \| v \| \qquad  \forall v \in E.
$$
We conclude that $b$ will satisfy \eqref{domin} for sufficiently large $t$ if
we choose $\e$ and $s$ so that $s \chi' + \e < \chi+s\chi $ while $s \chi - \e>0$.
This is equivalent to
$$ \frac{\e}{\chi} < s < \frac{\chi - \e}{\chi' -\chi}$$
and hence we can choose such $s$ if $\e$ is sufficiently small.  \QED
\end{proof}

\section{Normal forms}
\label{section:normalforms}

We consider the suspension action $\ta$ of $\rk$ on $\S$.
We fix a Lyapunov hyperplane $L \subset \rk$ and denote by $E$ and $\w$
the corresponding coarse Lyapunov distribution and foliation for $\ta$.

In this section we study properties of the action along the leaves of $\w$ and
introduce smooth coordinate changes along the leaves of $\w$ with respect
to which the elements act as certain polynomials. This method was introduced
to the study of local rigidity of higher rank abelian actions in \cite{KS97} and
uses the nonstationary normal forms of smooth contractions developed in \cite{GK,G}.
In contrast to the case of small perturbations of algebraic actions considered
in  \cite{KS97}, the action $\ta$ may not have the so-called  ``narrow band" property.
Instead of uniform growth estimates given by the narrow Mather spectrum,
we have to use nonuniform estimates given by the Multiplicative
Ergodic Theorem for the measure $\mu$. Therefore, the coordinate
changes will vary on $\S$ not continuously but measurably.

Let $a$ be an element in the negative Lyapunov half-space $L^- \subset \rk$,
so that $f = \ta (a)$ contracts $\w$. We will view it as a measure-preserving
system $(f,\mu)$. Its action along $\w$, $f : \w (x) \to \w (fx)$, defines
an extension  $\Phi:\S \times \Rm \to \S \times\Rm$ of $f$, where $m = \dim \w$.
Indeed, the leaf $\w (x)$ can be smoothly identified with the tangent space
$E(x)$, and the distribution $E$ can always be measurably trivialized on a set of full
measure. The
extension $\Phi_a$ preserves the zero section and acts by $\CI$ diffeomorphisms
in the fibers. In other words, $\Phi_a$ can be written in coordinates
$(x,t)\in \S \times\Rm$ as
$$\Phi_a(x,t)=(f(x), F_x(t)) $$
where $F_x(0)=0$ and $F$ is  $\CI$ in $t$.
We  will allow coordinate
changes which  are measurable in $x$, preserve each fiber $\Rmx$,  fix
the origin, are $\CI$ in each fiber, and have tempered  logarithms of
all derivatives of all orders at the zero section. We will call such
coordinate changes {\it admissible}. Recall that a real-valued function
$\varphi$ is called tempered with respect to the action $\ta$
if $\lim_{\, b \to\infty} \, \| b \|^{-1}\varphi(\ta(b)\,x)=0$ for $\mu$ - a.e. $x$.

The derivatives in the $t$ variable at the zero section  define a linear
extension of $f$, which we will denote  by $D_0 F_x$ and call the
derivative extension.  Note that $D_0 F_x$ are bounded functions on $\S$
and that this extension has negative Lyapunov exponents. Let
$\chi_1,\dots\,\chi_l$ be the different Lyapunov
exponents of the derivative extension and  $m_1,\dots,m_l$ be
their multiplicities.  Represent $\Rm$ as the
direct sum of the spaces  $\mathbb R^{m_i},\dots, \mathbb R^{m_l}$ and let
$(t_1,\dots,t_l)$ be the corresponding coordinate representation of
a vector $t\in\Rm$. Let $P: \Rm\to\Rm;\,\,(t_1,\dots,t_l)\mapsto
(P_1(t_1,\dots,t_l),\dots,P_l(t_1,\dots,t_l))$ be a polynomial
map preserving the origin. We will say that the map $P$ is of
{\it subresonance type} if it contains only such  homogeneous terms
in $P_i(t_1,\dots,t_l)$ with degree of homogeneity $s_j$ in the
coordinates of $t_j,\,\,i=1,\dots,l$ for which the  subresonance
relation $\chi_i \le \sum_{j\neq i}s_j\chi_j$ holds. There are only
finitely many subresonance relations and it is known \cite{G,GK} that
polynomial maps of the subresonance type with invertible derivative
at the origin generate a finite-dimensional Lie group. We will denote
this group by $SR_{\chi}$.

\begin{proposition} \label{normal form}

There exists an admissible coordinate change in $\S \times\Rm$
which transforms the extensions $\Phi_a$ for all $a \in L^-$ to
extensions $\Psi_a$ of the subresonance normal form
$$ \Psi(x,t)= (f(x),\mathcal P_x(t)) $$
where for almost every $ x\in X,\,\,\mathcal P_x\in SR_{\chi}$.

Moreover, this admissible coordinate change transforms into such
normal form any extension $\Gamma (x,t)=(g(x),\mathcal G_x(t))$
by $\CI$ diffeomorphisms preserving the zero section of a non-singular
transformation $g$ of $(\S,\mu)$ which commutes with $\Phi_a$ for
some $a \in L^-$.

\end{proposition}

\begin{proof}
We note that since $E$ is a coarse Lyapunov distribution, all Lyapunov
exponents of $\ta$ corresponding to $E$ are, by definition, positively
proportional. Therefore, the extensions $\Phi_a$ for all $a \in L^-$ are
contractions with the same subresonance relations.  The existence of
an admissible coordinate change for a single $a^\ast  \in L^-$ is given by
Theorem 6.1 in \cite{KaK}. Since $\Phi_a$ commutes with $\Phi_{a^\ast}$,
the ``centralizer theorem" \cite[Theorem 6.3]{KaK} yields that this coordinate
change brings any other $\Phi_a$, for $a \in L^-$,
to the subresonance normal form of $\Phi_{a^\ast}$. The coincidence of
resonances implies that this normal form is also the normal form for $\Phi_a$.
Then the ``centralizer theorem" can be applied to this coordinate change
with any $a \in L^-$ and yields the second part of the proposition.
 \QED
\end{proof}

\section{Commuting holonomies}  \label{commuting holonomies}

Let $\mathcal{W}$ be a coarse Lyapunov foliation as in
Section \ref{Smooth holonomies} with a complementary foliation
$\mathcal{W} _1$ as in Lemma \ref{splittings}. To simplify notations
in Sections \ref{commuting holonomies} and \ref{limit argument} we will
denote the corresponding foliations for the algebraic action $\tr$
by $\mathcal{W}^*$
and $\mathcal{W}^*_1$ respectively. In this section we will study the
holonomies along $\mathcal{W} _1$ between leaves of $\mathcal{W}$.
While in a general setting holonomy along a foliation is only a locally
defined operation, in our setting holonomies are realized by global homeomorphisms. Before specifying this we introduce some
notations and describe the algebraic foliations $\mathcal{W}^*$
and $\mathcal{W}^*_1$.

We have two actions  $\tilde \alpha$ and its linearization $\tilde \rho$
on the suspension manifold $\mathcal{S}$.  Recall that $\mathcal{S}$
is a homogeneous space $S / \Lambda$ where $S = \R^k \ltimes N$
is a solvable Lie group and $\Lambda = \Z^k \ltimes \Gamma$ is a
lattice in $S$. A left coset foliation is the foliation defined by orbits
of some subgroup $D < S$.  The foliations $\mathcal{W}^*$ and
$\mathcal{W}^*_1$, as well as other coarse Lyapunov and stable
foliations for the $\tilde \rho$ on $\mathcal{S}$, are left coset foliations.
This is most easily seen at the level of Lie algebras.
Let $\mathfrak s$ be the Lie algebra of $S$ and $\mathfrak n$ the
Lie algebra of $N$.  We can identify the tangent bundle of
$\mathcal{S}$ with $\mathcal{S} \times \mathfrak{s}$.
The fibration $N/\Gamma \rightarrow \mathcal{S} \rightarrow  \R ^k / \Z ^k$
of $N/\Gamma$ defines a foliation whose tangent bundle  is given by
$\mathfrak{n}$ in this identification.  Since $\tr$ is the suspension of the
action $\rho$ by affine automorphisms of $N/\Gamma$ any ``dynamical"
foliation as above is tangent to an invariant distribution given by a
subspace $\mathfrak{d} \subset \mathfrak{n}$, which by integrability
is a Lie subalgebra of $\mathfrak{n}$. This makes the corresponding
foliation into a left coset foliation for the subgroup $D<N$ such that
$Lie(D)=\mathfrak d$. For the coarse Lyapunov foliation $\mathcal{W}^*$,
and for the complementary foliation $\mathcal{W}^*_1$ we will denote
the corresponding nilpotent groups by $\bold{W}$ and $\bold{W}_1$.

Recall that by Lemma \ref{splittings} $\mathcal{W}^*$ and $\mathcal{W}^*_1$
subfoliate the leaves of $\mathcal{W}^* \oplus \mathcal{W}^*_1$, which is
a stable foliation for $\tr$. Moreover, on each leaf of this foliation
they form a global product
structure. This can be seen on the universal cover, which for the algebraic
action on $N/\Gamma$ can be identified with the Lie algebra $\mathfrak{n}$.
We choose any element $b \in \bold{W}_1$ and denote the translation
action of $b$ on $\mathcal{S}$ by $L_b (x) =b {\cdot} x=bx$.
Then for any such $b$ and
any $x$ in $\mathcal{S}$ the holonomy along $\mathcal{W}^*_1$ is
a diffeomorphism between $\mathcal{W}^*(x)$ and $\mathcal{W}^*(bx)$,
which we denote by $h^*_{b,x}$.

Similarly, for any $b$ in $\bold{W}_1$ and any $x$ in $\mathcal{S}$ we denote
by $h_{b,x}$ the holonomy along $\mathcal{W}_1$ between $\mathcal{W}(x)$
and $\mathcal{W}(bx)$. Since the conjugacy $\tp$ maps $\mathcal{W}$ to
$\mathcal{W}^*$ and $\mathcal{W}_1$ to $\mathcal{W}^*_1$  we see that
$h_{b,x}$ is a global homeomorphism and $h_{b,x}^* \circ \tp = \tp \circ h_{b,x}$.
Moreover, $h_{b,x}$ is a diffeomorphism by Corollary \ref{smooth}.

In order to use Proposition \ref{normal form} we can use the holonomies
$h_{b,x}^*$ and $h_{b,x}$ to define bundle maps in the following manner.  We take a
bundle ${\mathcal S} \times \bold{W}$ with base $\mathcal S$ and
fiber $\bold{W}_s$, which we think of as the leaf of $\mathcal{W}^*$
or $\mathcal{W}$ through $s$.
For the leaves of $\mathcal{W}^*$ we
have a natural identification of $\bold{W}_s$ with $\mathcal{W}^*(s)$
given by left translations: $w \mapsto w\cdot s$. For $\mathcal{W}$
we fix some smooth identification that depends continuously on $s$
in $\ci$ topology on compact subsets of $\bold{W}$. The holonomy
$h_b^*$ can now be viewed as a bundle map
$h^*_b: \mathcal{S} \times \bold{W} \rightarrow  \mathcal{S} \times \bold{W}$
covering the left translation $L_b$ in the base, where $h^*_b(x,w)=h^*_{b,x}(w)$.

Similarly, we define $h_b$ via the equation $h_b(x,w)=h_{b,x}(w)$ and $h_b$ 
is a bundle map $h_b : \mathcal{S} \times \bold{W} \rightarrow  \mathcal{S} \times \bold{W}$
which is smooth along the fibers and covers the homeomorphism
$\phi^{-1} \circ L_b \circ \phi$ of $\S$. Note that this homeomorphism
preserves the invariant measure $\mu$ on $\S$ since $L_b$
preserves the Lebesgue measure $\lambda = \tp _* (\mu)$.
Since the nilpotent
group $\bold{W}$ is diffeomorphic to $\R^m$, we are in the setup of
Proposition \ref{normal form}.
The actions $\ta$ and $\tr$
also lift naturally to the corresponding actions on the bundle
$\mathcal{S} \times \bold{W}$. Slightly abusing notations we will denote
the lifts by the same letters. Since we do not know the smoothness of
$\mathcal{W}$, we can only say that the lift of $\ta$ is smooth along the
fibers. Note that the natural extension of $\tp$ to $\mathcal{S}\times\bold{W}$
conjugates the lifts of the actions as well as the holonomy maps $h_b$
and $h_b^*$.

We will use the
algebraic structure of $\tr$ to show that $h_b^*$ commute with certain
elements of $\tilde \rho(\R^k)$ and then use this to conclude that
$h_b$  commute with certain elements of  $\tilde \alpha(\R^k)$.
The main goal of this section is to prove the following theorem.

\begin{theorem}
\label{lemma:normalformholonomyfullstable}
If the action $\rho$ is semisimple, then for every $b \: {\in} \:\bold{W}_1$
the maps $h_b$ for all $ b \in \bold{W}$ preserve, $\mu$-almost everywhere,
a fixed normal form along leaves of $\mathcal{W}$.
\end{theorem}
\vskip.2cm

\begin{proof}
We begin by finding subfoliations of $\mathcal{W}_1$ for which the
holonomy commutes with some $\ta (v)$, $v \in \mathbb R^k$, contracting
$\mathcal{W}$. To do this we will work with the algebraic action $\tr$.
Recall that $T\mathcal{W}^*$ splits as a sum of coarse
Lyapunov distributions $\oplus E^j$. Let $E'$ be one of these distributions
and $L'$ be the corresponding Lyapunov hyperplane. Let $v$ be any element
of $L'$ for which $\tr (v)$ contracts $\mathcal{W}^*$. Then $\tr (v)$ acts
isometrically on a certain foliation $\mathcal{H}^*_{v}$ which is defined as
the orbits of the action of some subgroup $\mathbf {H}_{v}$ in $\bold{W}_1$.
 Since
$\tr (v)$ is semisimple by the assumption, $\mathcal{H}^*_{v}$ is in fact the full
coarse Lyapunov foliation of $\tr$ corresponding to $L'$. (If the derivative of
$\tr (v)$ on $E'$ had Jordan blocks, this would be a strict subfoliation.)
Now we can decompose the Lie algebra of  $\mathbf{H}_{v}$ into the
irreducible subspaces of the rotation defined by taking the skew symmetric
part of $\tr (v)$. We denote the resulting Lie subgroups of  $\mathbf{H}_{v}$
by $\mathbf{H}_{v,i}$.

\begin{lemma}
\label{lemma:foliationwithtrivialvaction}
For any element $v \in L'$ there are real numbers $t_i>0$
 such that for any $b \: {\in} \: \bold{H}_{v,i}$ the map $h^*_b$ commutes
 with $\tilde \rho(t_i v)$.
\end{lemma}

 \begin{proof}
 A suitable multiple $t_i \, v$ of $v$ commutes with the group $\mathbf{H}_{v,i}$.
Hence translations by elements of $\bold{H}_{v,i}$  commute with $\tr(t_i v)$.
Then the holonomy $h_b^*$ will also commute with $\tr(t_i v)$ since it agrees
with holonomy along the foliation $\mathcal{H}^*_{v,i}$.  \QED
\end{proof}

Since $\tp$ conjugates the actions and the holonomies this lemma yields that
for any $b \: {\in} \: \bold{H}_{v,i}$ the map $h_b$ commutes with $\ta (t_i v)$.
By making different choices of the Lyapunov hyperplane $L'$, $v\in L'$, and $i$,
we can arrange so that the groups $\bold{H}_{v,i}$ generate $\bold{W}_1$.
Proposition \ref{normal form} implies that there is a common normal form for
 $h_b$ for all  $b$ in all $\bold{H}_{v,i}$. Therefore, $h_b$ with any $b$ in
$\bold{W}_1$, which is a composition of maps of this form, also preserves this normal form.
This completes the proof of Theorem \ref{lemma:normalformholonomyfullstable}.
\QED
\end{proof}

Now we give a more detailed description of the algebraic holonomies
$h_b^*$.  As a corollary we describe certain limits of the maps
$h_b^*$ and $h_b$ which will be used in the next section.

\begin{proposition}
\label{proposition:holonomiesareisometric}
For any $b \: {\in}\: \bold{W}_1$ and any $x \in \S$, the holonomy 
$h^*_{b,x} : W^* (x) \to W^* (b x)$ is equivariant  with respect to
the action of $\bold{W}$ along leaves of $\mathcal{W}^*$.
\end{proposition}

\begin{proof}
We need to show that the holonomy along $\mathcal{W}^*_1$ commutes
with the action of $\bold{W}$ along leaves of $\mathcal{W}^*$.
First we observe that $\bold{W}$ normalizes $\bold{W}_1$.
To see this we note that there is a subgroup $\bold{W'}=\bold{W}\bold{W}_1$
in $N$.  This is the group that corresponds to the foliation
$\mathcal{W}^* {\oplus} \mathcal{W}^*_1$ as in Lemma \ref{splittings}.
We denote the Lie algebras of $\bold{W}$ and $\bold{W}_1$ by
$\mathfrak w$ and $\mathfrak w_1$. To conclude that $\bold{W}_1$ is
normal inside $\bold{W'}$ we choose an element $s \in \R^k$ for
which $\tr(s)$ acts isometrically on $\mathfrak w$ and
 contracts exactly $\mathfrak w_1$.  This is possible by
the construction of $\mathcal{W}^*$ and $\mathcal{W}^*_1$ in Lemma \ref{splittings}.  Then any bracket $[v,u]$ where  $u \in \mathfrak w$ and $v \in \mathfrak w_1$ is contracted by $\tr (s)$ and hence must be in $\mathfrak w_1$.

 Now let $a \in \bold{W}$ and $b  \in \bold{W}_1$.  For any $x \in \S$
we can write $a b  x = a b a^{-1} a  x$.  By the normalization we have
$a b a^{-1} \in \bold{W}_1$. Hence the point $a b x$ is both in
$\mathcal{W}^*(bx)$ and in $\mathcal{W}^*_1(a x)$.  This shows that
$h_{b,x}^* (ax) = abx$ for any $x \in \S$ and $b  \in \bold{W}_1$ and
proves that the holonomy $h_{b,x} ^*$
 from $\mathcal{W}^*(x)$  to $\mathcal{W}^*(bx)$ commutes with the
 action of $\bold{W}$.
\QED
\end{proof}

\begin{corollary}
\label{corollary:linearlimits}
Suppose that for some elements $b_n \in \bold{W}_1$ and some point
$x \in \S$ the sequence $b_n x$ converges to a point $y$ in $\mathcal{W}^*(x)$.
Then the holonomy maps $h_{b_n,x} ^*$ 
converge to the diffeomorphism
$T_{x,y}^* : \mathcal{W}^*(x) \to \mathcal{W}^*(x)$
given by $T_{x,y}^* (ax) =a y $ for any  $a \in \bold{W}$.
Consequently, the holonomy maps $h_{b_n,x} $ 
 converge to the homeomorphism
$\tp ^{-1} \circ T_{x,y}^* \circ \tp \;$ of the leaf $\mathcal{W}(\tp ^{-1}(x))$
uniformly on compact sets.
\end{corollary}

\begin{proof}
By Proposition \ref{proposition:holonomiesareisometric}, if $a \in \bold{W}$
then $h_{b_n,x} ^* (a x) = a h_{b_n,x} ^* (x)  = a (b_n x)$, which converges
to $a y$ as desired. The first claim now follows since
$\mathcal{W}^*(x)=\bold{W}x$.
Conjugating by $\tp$ gives the second claim.
\QED
\end{proof}

\noindent{\bf Remark:} It is not clear that the limit of the $h_{b_n,x}^*$
can be realized as a holonomy of any kind along any leaf from $W^*(x)$
to $W^*(y)$.

\section{Limiting Argument} \label{limit argument}
The main goal of this section is to prove the following proposition,
which we will then use to complete the proof of Theorem \ref{nilmanifold}.
We retain the notations of the previous section.

\begin{proposition}
\label{smoothalongleaves}
For $\mu$-almost every $x \in \S$ there are smooth transitive actions
of $\bold{W}$ on the leaves $\mathcal{W}(x)$ and $\mathcal{W}^*(x)$
which are  intertwined by the conjugacy $\tp$.
\end{proposition}

\begin{proof}
For any given $x \in \S$ we can naturally identify $\bold{W}$ with
$\mathcal{W}^*(x) =\bold{W}x $ by $w \mapsto wx$. In this identification,
we take the desired transitive action of $\bold{W}$ on $\mathcal{W}^*(x)$
to be the action by right translations. Corollary \ref{corollary:linearlimits}
means that the limits of the holonomy maps $h_{b_n,x} ^*$ are exactly
of this form. In fact, since $\mathcal{W}^*_1(x)$ is dense in the
corresponding $N/\Gamma$-fiber of the suspension as a full stable foliation of
Anosov element, one can see that any right translation can be obtained
as such a limit. While we will not use it directly, this provides motivation
for the argument.
The desired action of $\bold{W}$ on $\mathcal{W}^*(x)$ is
obtained by conjugating by $\tp$. This action is a priori only by
homeomorphisms and the goal is to prove that it is smooth.
For this we will study the limits of the holonomy maps $h_{b_n,x}$.

We consider Lusin sets where the measurable normal form on
the leaf $\mathcal{W}(x)$ depends continuously on $x$.
Let $\Lambda _m '$ be an increasing  sequence of such sets
with $\mu (\Lambda _m ') \to 1$. Let $\Lambda _m$ be the set
of  density points of $\Lambda _m '$, then $\mu (\Lambda _m) \rightarrow 1$.   Then there exists a subset $X \subset \Lambda = \cup \Lambda _m$
with $\mu (X)=1$ such that for all $x \in X$ the intersection
$\mathcal{W}(x) \cap \Lambda $ has full measure  with respect to the conditional measure of $\mu$ on $\mathcal{W}(x)$.

Fix any $x \in X$. Then for almost every $y$ in $\mathcal{W}(x)$
with respect to the conditional measure $x$ and $y$ belong to some
$ \Lambda _m$.   We pick a sequence $b_n$ of elements
in $\bold{W}_1$ with the following properties:
\begin{enumerate}
\item $x_n =b_n x \rightarrow y$
\item $x_n \in \Lambda_m$.
\end{enumerate}
To find such $b_n$ we use the fact that $y$ is a density point of $\Lambda_m$
and the fact that $\bold{W}_1$ acts ergodically with respect to $\mu$ on the
corresponding $N/\Gamma$-fiber of the suspension.
The ergodicity follows since the foliation $\mathcal{W}_1$ of $N/\Gamma$
is a full stable foliation of some Anosov element of $\a$  and hence is
uniquely ergodic by Bowen and Marcus \cite{BM}. Alternately, since
the push forward of $\mu$ by $\phi$ is Lebesgue, the ergodicity can
be checked on the algebraic side using the work of Auslander, Hahn,
and Green \cite{AGH}.

Each map $h_{b_n,x} $ is smooth and preserves the normal forms at
$x$ and $x_n$. By Corollary \ref{corollary:linearlimits} the sequence
$h_{b_n,x} $ converges to a homeomorphism
$T_{x,y}:  \mathcal{W} \to  \mathcal{W}$
conjugate by $\tp$ to the translation
$T^*_{\tp(x),\tp(y)}$ of $\mathcal{W}^*(\tp(x))$.
Since the normal form coordinates depend continuously on the base point
in $\Lambda_m$ and the maps $h_{b_n,x}$ in these coordinates belong to a fixed
Lie group, the limit  $T_{x,y}$ is smooth. Recall that the push forward of
$\mu$ by $\tp$ is the Lebesgue measure $\lambda$ and hence the
conditional measure of $\mu$ on $\mathcal{W}(x)$ is mapped by $\tp$
to the conditional measure of $\lambda$ on $\mathcal{W}^*(\tp(x))$,
which is equivalent to volume on $\mathcal{W}^*(\tp(x)) =\bold{W}\tp(x) $.
We conclude that for almost every element of $\bold{W}$ the
corresponding translation is conjugate by $\tp$ to a diffeomorphism of
$\mathcal{W}(x)$. Hence the subgroup of $\bold{W}$ that acts by
diffeomorphisms of $\mathcal{W}(x)$ has full measure and must be
the whole $\bold{W}$ by the next lemma. It now follows from
\cite[Section 5.1, Corollary]{Montgomery-Zippin} that the action of
$\bold{W}$ on $\mathcal{W}(x)$ is smooth. This completes the proof of
Proposition \ref{smoothalongleaves}.\QED
\end{proof}

\begin{lemma}
Let $G$ be a Lie group.  Then  any subgroup $H$ of full measure is $G$.
\end{lemma}

\begin{proof} If not then the distinct cosets of $H$ in $G$ are disjoint sets of full measure which is impossible. \QED
\end{proof}

\noindent{\bf Remark:} It is possible to prove that $G$ is smooth along a generic leaf of $\mathcal{W}$ using
older methods involving returns along Weyl chamber walls in $\R^k$ instead of holonomies.  However, one cannot obtain uniformity in estimates this way nor complete the proof below without using holonomies.

\vskip.3cm

{\em End of Proof of Theorem \ref{nilmanifold}}:
We need to show that $\phi$ is a diffeomorphism.
It will be easier to work with $\phi ^{-1}$ as we will employ certain
elliptic operators defined by right invariant vector fields to prove
smoothness of $\phi ^{-1}$.

Proposition \ref{smoothalongleaves} implies that for any coarse Lyapunov
foliation $\tp^{-1}$ intertwines transitive $\ci$ group actions on typical
leaves $\mathcal{W}(x)$ and $\mathcal{W}^*(x)$ in the suspension $\S$.
This yields that, for a typical $x$ in $M=N/\Gamma$, $\, \phi^{-1}$ intertwines
transitive $\ci$ group actions on $\mathcal{W}(x)$ and $\mathcal{W}^*(x)$.
Hence $\phi^{-1}$ is $\ci$ along $\mathcal{W}^*(x)$.

We claim that $\phi^{-1}$ is $\ci$ along {\em all} leaves of $\mathcal{W}^*$
and that all its derivatives along the leaves are continuous on $M$.
This follows from the fact that $TM= T\w \oplus T\w_1\oplus T\w_2$
and that the holonomies between different leaves of $\w$ along $\w_1$
and $\w_2$
are smooth and intertwine the restriction of $\phi^{-1}$ to these leaves.

We can now finish the proof quickly.  We know that $\phi ^{-1}$ is smooth along the coarse Lyapunov foliations with continuous dependence of  the derivatives.  This simply says that derivatives of all orders exist and are continuous for each right invariant vectorfield tangent to a coarse Lyapunov foliations (while mixed derivatives may fail to exist).  Pick a basis of such vectofields $X_i$.  Then $X^l:= \sum _i  X_i ^{2l}$ for any $l$ is an elliptic operator of order 2l.  It follows that $X^l (\phi ^{-1})$ is smooth for all $l$. Hence by elliptic theory, $\phi ^{-1}$ is $\ci$.  We refer to [5, Section 7.1]  e.g. for a more detailed discussion of this elliptic theory argument.

 It remains to show that $\phi^{-1}$ is a diffeomorphism. Since $\phi ^{-1}$ is already a homeomorphism,
this follows once we show that the differential of $\phi^{-1}$ is everywhere non-degenerate.   This  follows easily from
Proposition \ref{mu}.  Indeed,  we have $\mu = \phi^{-1} _\ast (\lambda)$
and $\mu$ has smooth positive density.
\QED

 \section{Totally reducible  actions and examples.}
\label{section:generalresults}

Here we will prove Corollary 1.2.  By the proposition below, this is  immediate from Theorem 1.1.

Recall that an   algebraic  $\mathbb Z^k$ action on a torus  is called {\em irreducible}
if there is no rational invariant subtorus, and {\em totally reducible} if
every rational invariant subtorus has a rational invariant complement.

Given a nilmanifold $N/\Gamma$, there is a maximal toral quotient $\mathbb
T^d$ obtained by taking $N/[N,N]\Gamma$.  Any action by
automorphisms on $N/\Gamma$ descends to an action on $\mathbb T^d$,
which we refer to as the {\em maximal toral quotient action}.
We say
that an   algebraic  $\mathbb Z^k$ action on $N/\Gamma$ is {\em totally
reducible} if the maximal toral quotient action is totally reducible
and there is a $\Z^k$  invariant complement to $[\mathfrak n, \mathfrak n]$
in the Lie algebra $\mathfrak n$ of $N$.

 We call an action by affine automorphisms of a nilmanifold
{\em totally reducible}
if the finite index subgroup that acts by automorphisms is totally reducible.

It is easy to see that semisimple actions are totally reducible.

\begin{Proposition}
A totally reducible $\Z^k$ action on a nilmanifold is semisimple.
\end{Proposition}

\begin{proof} First we consider an irreducible torus action.  Let $A$ be a toral
automorphism, i.e. an integral matrix.
The characteristic polynomial of $A$ splits over
$\Q$ as $\prod P_i(X)^{d_i}$.  Then the kernel $E(A)$ of $\prod P_i(A)$
is the subspace spanned by the eigenspaces of $A$.  It is rational as
the kernel of a rational operator.

If a collection $A_i$ of toral automorphisms commute then
$E(A_1)$ is invariant under $A_2$.  Consider  the restriction $B_2$ of $A_2$ to $E(A_1)$  Then $E(B_2)$ is nontrivial, and contained in
$E(A_1) \cap E(A_2)$.  Inductively we see that  $\cap E(A_i)$ is nontrivial.  Thus we get a nontrivial rational subspace invariant under all $A_i$.
This defines an invariant proper subtorus unless all $A_i$ are semisimple.
Hence irreducible torus actions are semisimple.

Considering irreducible components of totally reducible torus actions
it follows easily that they are also semisimple.

Finally consider a totally reducible action on a nilmanifold.  Then the maximal toral quotient action is totally reducible and hence semisimple.  This implies that the action on the invariant complement $\R^d$ to $[\mathfrak n, \mathfrak n]$
is semisimple.  Since joint eigenvectors for $\Z^k$ span $\R^d$, their brackets, which are also eigenvectors
span $\mathfrak n$.  Therefore the action is semisimple.
\QED
\end{proof}

We briefly describe many examples of totally irreducible Anosov actions on nilmanifolds.  These examples are
more general variants of examples constructed by Qian in \cite{Qi}.  Let $\T^d$ be a torus with an Anosov  algebraic semisimple $\Z^k$ action.  The action lifts to the vector space $\R^d$.  Let $N=N^k(\R^d)$ be the $k$-step free
nilpotent Lie group generated by $\R^d$.  (It is somewhat more typical to define this at the level of Lie
algebras, but the meaning is clear as long as we assume $N^k(\R^d)$ is simply connected.)  The $\Z^k$ action on $\R^d$
extends canonically to a $\Z^k$ action on $N^k(\R^d)$ and preserves the obvious rational structure on that group.
This implies that we have a well-defined $\Z^k$ action on $N/\Gamma$ where $\Gamma$ is a lattice in $N$.

It is easy to check that generically this construction takes an Anosov $\Z^k$ action on $\T^d$ and lifts it to an Anosov action on $N/\Gamma$.  An Anosov automorphism $A$ of $\T^d$ lifts to an Anosov automorphism of $N/\Gamma$ as long
as no product of length at most $k$ of eigenvalues of $A$ has modulus one.  It is straightforward to construct
many examples which are also $TNS$ using similar algebraic condition on eigenvalues.

We remark that the hypothesis of Theorem \ref{nilmanifold} are necessary for our argument as there are examples for which the commuting holonomies are not ergodic.

\begin{example}
Take a semisimple Anosov linear action of $\Z^k$ on $\T^d$, we can define an action on $\T^{2d}$ by letting $A\in \Z^k$ act by
$A(x,y)=(Ax, Ay + x)$.  It is straightforward to check that for examples of this kind, the commuting holonomies
are not ergodic.
\end{example}






\bibliographystyle{alpha}

\end{document}